\documentclass[a4paper,11pt]{article}

\usepackage[ansinew]{inputenc}
\usepackage{amsmath}
\usepackage{amsfonts}
\usepackage{amssymb}
\usepackage{amsthm}
\usepackage{graphicx}
\usepackage{color}
\usepackage{multirow}
\usepackage{enumerate}
\usepackage[Lenny]{fncychap}
\usepackage{microtype}
\usepackage{fancyhdr}
\usepackage[all]{xy}
\usepackage{url}
\usepackage{varioref}
\usepackage{fancyvrb}
\usepackage{listings}
\usepackage{float}
\usepackage{eso-pic}

\lstdefinelanguage{Magma}
{
keywords={for,end,if,then,while,function,return,cat,&,and,or },
morekeywords={Seqset,Setseq,Polytope,AutomorphismGroup,RowSequence,IdentifyGroup,
	      Subgroups,PermutationMatrix,Generators,MatrixGroup,Transpose},
sensitive=false,
morecomment=[l]{//},
morecomment=[s]{/*}{*/},
morestring=[b]",
}

\lstnewenvironment{codice_magma}[1][]
{\lstset{basicstyle=\small\ttfamily, columns=fullflexible,
language=Magma,
keywordstyle=\color{red}\bfseries,
commentstyle=\color{blue},
numbers=left, numberstyle=\tiny,
stepnumber=2, numbersep=5pt, frame=single, #1}}{}

\DeclareGraphicsExtensions{.pdf,.png}

\newtheoremstyle{mystyle}
  {}
  {}
  {}
  {}
  {\bfseries}
  {}
  {\newline}
  {}

\newtheorem{thm}{Theorem}
\newtheorem{lem}[thm]{Lemma}

\newtheorem{prop}[thm]{Proposition}
\newtheorem{coroll}[thm]{Corollary}
\newtheorem{remark}[thm]{Remark}

\newcommand{\C}{\mathbb{C}}
\newcommand{\R}{\mathbb{R}}
\newcommand{\Q}{\mathbb{Q}}

\newcommand{\Z}{\mathbb{Z}}

\DeclareMathOperator{\id}{Id}
\DeclareMathOperator{\h}{h}

\DeclareMathOperator{\Ker}{Ker}

\DeclareMathOperator{\Aut}{Aut}

\DeclareMathOperator{\Pic}{Pic}

\DeclareMathOperator{\Fix}{Fix} 
\DeclareMathOperator{\e}{e}

\renewcommand{\P}{\mathbb{P}}

\renewcommand{\H}{H}

\newcommand{\spn}[1]{\left < #1 \right >}

\newcommand{\ti}[1]{\tilde{#1}}
\newcommand{\du}[1]{#1^o}
\newcommand{\Ch}{Ch}
\newcommand{\A}{A}
\newcommand{\D}{D}
\newcommand{\Spec}{Spec}

\DeclareMathOperator{\age}{age}

\title{A Closer Look at Mirrors and Quotients of Calabi-Yau Threefolds}
\author{Gilberto Bini, Filippo Francesco Favale}
\date{\today}

\begin{document}

\maketitle
\begin{abstract}
\noindent Let $X$ be the toric variety $(\P^1)^4$ associated with its four-dimensional polytope $\Delta$. Denote by $\ti{X}$ the resolution of the singular Fano $\du{X}$ variety associated with the dual polytope $\du{\Delta}$. Generically, anticanonical sections $Y$ of $X$ and anticanonical sections $\ti{Y}$ of $\ti{X}$ are mirror partners in the sense of Batyrev. Our main result is the following: the Hodge-theoretic mirror of the quotient $Z$ associated to a maximal admissible pair $(Y,G)$ in $X$ is not a quotient $\ti{Z}$ associated to an admissible pair in $\ti{X}$. Nevertheless, it is possible to construct a mirror orbifold for $Z$ by means of a quotient of a suitable $\ti{Y}$. Its crepant resolution is a Calabi-Yau threeefold with Hodge numbers $(8,4)$. Instead, if we start from a non-maximal admissible pair, in same cases, its mirror is the quotient associated to an admissible pair.
\end{abstract}


\section{Introduction}

\noindent Let $T$ be a Calabi-Yau threefold, i.e., a compact K\"ahler manifold with trivial canonical bundle and no holomorphic $1$-forms. As an example, take a generic anticanonical section $Y$ of a smooth toric Fano fourfold $X$. Following Batyrev's seminal article (\cite{Bat94}), there exists a mirror partner of $Y$, which is given as follows. The toric fourfold $X$ is associated with a polytope in four-dimensional real space. The dual polytope yields another toric variety $\du{X}$, which is not smooth in general. Take a toric resolution $\ti{X}$ associated to a maximal projective subdivision the fan of $\du{X}$. A generic anticanonical section $\ti{Y}$ of $\ti{X}$ is a mirror partner of a generic anticanonical section of $X$. Here by a mirror partner we mean only a Hodge-theoretic mirror partner, i.e., the relevant Hodge numbers of $Y$ and $\ti{Y}$ are exchanged, namely:
$$
h^{1,1}(Y)=h^{2,1}(\ti{Y}), \quad h^{1,1}(\ti{Y})=h^{2,1}(Y).
$$

\noindent Let $G$ be a finite group of automorphisms of $X$ that acts freely on $Y$. It is easy to check that $Z:=Y/G$ is a Calabi-Yau threefold with non-trivial fundamental group. It is natural to ask whether, for each pair $(Y,G)$, there exists a finite group $\ti{G} \subset \Aut(\ti{X})$ such that $\ti{Z}:=\ti{Y}/\ti{G}$ is a mirror partner of $Z$ - at least Hodge-theoretically. 
\vspace{2mm}

\noindent In Section (\ref{SEC:ANTI}), we give a negative answer to the question above if the groups acts with the highest possible order. More precisely, let $X=(\P^1)^4$ and consider pairs $(Y,G)$, where $Y$ is a Calabi-Yau threefold in $X$ and $G$ is a finite group of automorphisms of $X$ that acts freely on $Y$ with the maximum possible order. These pairs are first investigated in (\cite{BFNP13}), where they are called {\em maximal admissible pairs}.  For each of them there exists a Calabi-Yau threefold $Z=Y/G$ with non-trivial fundamental group and relevant Hodge numbers $(\h^{1,1},\h^{1,2})=(1,5)$. A (Hodge-theoretic) mirror partner of $Z$ should have 

\begin{equation}
\label{zedmirror}
h^{1,1}(\ti{Z})=5, \quad h^{1,2}(\ti{Z})=1
\end{equation}
and height $6$, where the height is defined as the sum of the two Hodge numbers above. 
\vspace{2mm}

\noindent In this setting, $X$ is the toric Fano fourfold associated with the hypercube. Its dual is the $16$-cells (or hyperoctahedron) that yields a singular toric variety. If we take a maximal projective subdivision of the hypercube, we have a smooth toric variety, which is not Fano because the anticanonical bundle is semi-ample and not ample; however, this resolution is crepant. In other words, an anticanonical section $\ti{Y}$ of the resolution is a smooth Calabi-Yau, which is a mirror of $Y$. 
\vspace{2mm}

\noindent In order to find subgroups of $\Aut(\ti{X})$, we recall the structure of the automorphism group of a toric variety (cfr. \cite{Cox95}). This is given by the combinatorics of the fan, as well as by the dense torus in $\ti{X}$. The former yields a unique possible group in order to satisfy (\eqref{zedmirror}); the latter helps us to describe all possible families of finite groups acting on $\ti{Y}$.  None of them acts without fixed points; furthermore, the fixed locus is the union of some curves. Therefore, a mirror - if it exists - can not be found in this way. Of course, the quotients by the groups described above have finite quotient singularities, and there exists a mirror Calabi-Yau orbifold that satisfies (\eqref{zedmirror}).
\vspace{2mm}

\noindent If we take an crepant resolution of the orbifold mentioned before, we get a smooth Calabi-Yau manifold with height $12$ and $(\h^{1,1},\h^{1,2})=(8,4)$. This is a mirror partner (at least Hodge-theoretically) of a quotient in (\cite{BF11}), which is a Calabi-Yau manifold with cyclic fundamental group of order four: see Remark (\ref{osette}).
\vspace{2mm}

\noindent Finally, in Section (\ref{SEC:MIR}), we prove that it is indeed possible to find Hodge-theoretic mirror pairs by taking quotients associated to admissible pairs in $X$ and in $\ti{X}$ if the admissible pair in $X$ is not maximal. More precisely, we construct three pair $(Z_i,\ti{Z}_i)$ of Hodge-theoretic Calabi-Yau mirrors with $Z_i$ associated to an admissible pair in $X$ and $\ti{Z}_i$ associated to an admissible pair in $\ti{X}$.
\vspace{2mm}

\noindent {\bf Acknowledgments.} This work is part of the second named author's Ph.D thesis. We are grateful to the anonymous referee of the thesis for suggestions and helpful remarks, especially on the chapter related to the material presented here. The first named author is partially supported by FIRB 2012 "Moduli spaces and Applications" and PRIN 2010-2011 "Geometry of Algebraic Varieties". Both authors are partially supported by MIUR and GNSAGA.

\section{Automorphisms of Resolutions of the Toric Dual of $(\P^1)^4$}
\label{SEC:AUTORES}
\noindent Let $X$ be $(\P^1)^4$. $X$ is a Fano fourfold and it is the toric variety associated to the polytope in $4$-dimensional real space $\Delta=[-1,1]^4$, the hypercube. The fan $\Sigma$ of $X$ is spanned by the faces of the dual polytope $\du{\Delta}$ of $\Delta$, i.e., the $16-$cells that can be realized as the convex hull of the points $\{\pm \e_i\}_{i=1,\dots,4}$, where $e_i$ form the canonical basis of ${\mathbb R}^4$.
\begin{figure}[H]
\begin{center}
\includegraphics[width=.8\textwidth]{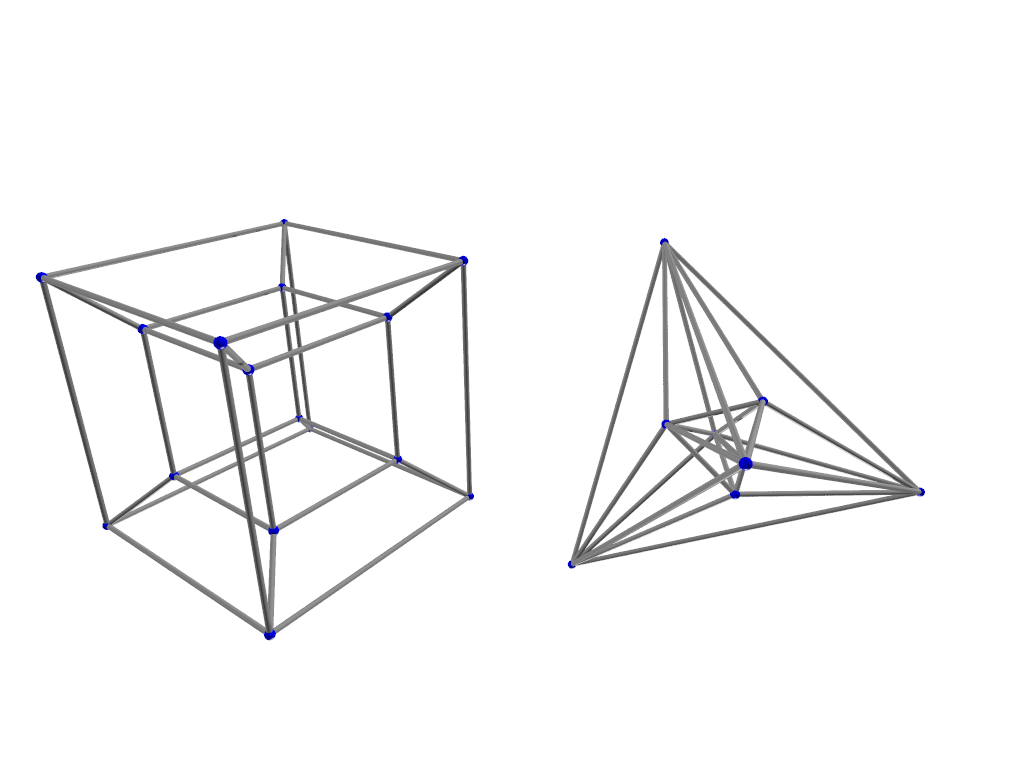}
\caption{A representation of a hypercube and its dual: the $16-$cells.}
\end{center}
\end{figure}

\noindent As shown in (\cite{BF11}) and (\cite{BFNP13}), there exists a smooth and $G-$invariant anticanonical section $Y$ with $G$ of order $16$ that acts freely on $Y$. 
\vspace{2mm}

\noindent The Fano variety dual to $X$ is associated to the fan spanned by the faces of $\Delta$, i.e., the hypercube. Unfortunately, it is a strongly singular variety: indeed, its fan $\du{\Sigma}$ is neither smooth nor simplicial. For example, the cone whose primitive generators are $$(1,1,1,1)\quad(1,1,1,-1)\quad(1,1,-1,1)\quad(1,1,-1,-1)$$ is not simplicial and the same is true for all the cones of dimension $3$ and $4$. In order to describe the mirror Calabi-Yau of $Y$, it is necessary to choose a maximal projective subdivision $\ti{\Sigma}$ of $\du{\Sigma}$. By definition, every maximal projective subdivision fulfills the following conditions:
$$\ti{\Sigma}(1)=(\Z^4\cap \du{(\du{\Delta})})\setminus \{0\} = (\Z^4\cap \Delta)\setminus \{0\}.$$
There are $81$ integral points in $\Delta$: $16$ vertices of the hypercube (which correspond to the rays of $\Sigma$), $32$ points that are in the middle of an edge, $24$ that correspond to centers of $2-$faces (squares), $8$ that are centers of one of the $8$ cubes of the hypercube;
\begin{figure}[H]
\begin{center}
\includegraphics[width=.8\textwidth]{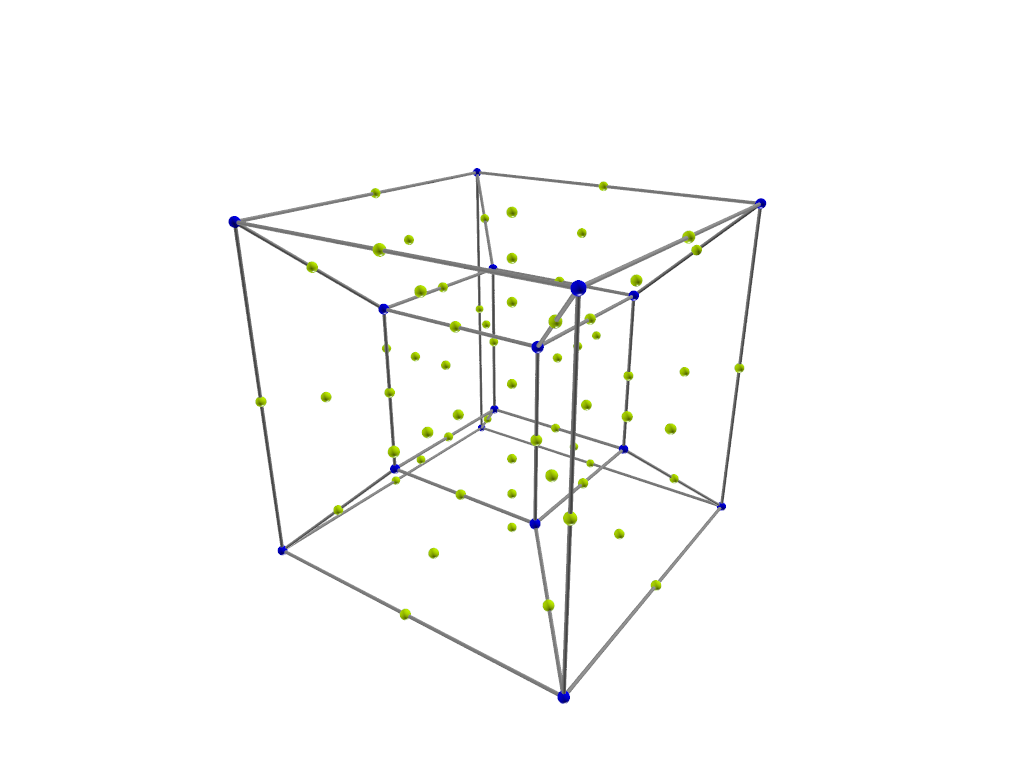}
\caption{Integral points in the hypercube.}
\end{center}
\end{figure}
\noindent Finally, we have the center of the hypercube which is the origin. This implies that every maximal subdivision of $\du{\Sigma}$ will have $80$ rays.
\vspace{2mm}

\noindent Given a toric variety $X$ associated to a fan $\Sigma$, we recall that its homogeneous coordinate ring $S$ is the ring $\C[x_{\rho}\, : \, \rho \in \Sigma(1)]$ with a particular grading on $\A_{n-1}(X)$. If $D=\sum_{\rho}a_{\rho}D_{\rho}\in \A_{n-1}(X)$, we define for brevity $x^{D}:=\Pi_{\rho}x_{\rho}^{a_{\rho}}$. The grading on $S$ is given by the condition 
$$\deg(x^{D})=\deg(x^{E})\Longleftrightarrow D=E\, \mbox{ in }\, \A_{n-1}(X).$$
One of the most important thing about $S$ is that it only depends on the rays of the fan $\Sigma$. Different fans with the same rays give the same homogeneous coordinate ring. We recall that a root of $X$ is a pair $(x_{\rho},x^D)$ such that $x_{\rho}$ is different from $x^D$ but they have the same degree. For each root it is possible to find a $1-$parameter group of automorphisms of the corresponding toric varieties. By definition, another family of automorphisms of toric varieties are those coming from the torus $T$. If $X$ is simplicial, there is only another elementary type of automorphisms, those obtained by the symmetries of the fan $\Sigma$. We denote this group by $\Aut(N,\Sigma)$, where $N$ is the lattice such that $N\otimes\R=\R^4$ is the vector space where the fan lies. For simplicial varieties there is an exact sequence 
\begin{equation}
\label{EQ:DEMAZURE}
1\rightarrow \Aut(X)_0\rightarrow \Aut(X)\rightarrow \frac{\Aut(N,\Sigma)}{\Pi_{i}S_{i}}\rightarrow 1 
\end{equation}
due to Demazure (smooth case) and Cox (simplicial case). $S_i$ is a permutation group on the partition of $\{\rho\}_{\rho\in \Sigma(1)}$ given by the equivalence relation $\rho_1 \sim \rho_2$ iff $\deg{x_{\rho_1}}=\deg{x_{\rho_2}}$.
\begin{lem}
Assume that $\ti{\Sigma}$ is a maximal projective resolution of $\du{\Sigma}$ and call $\ti{X}$ the corresponding toric variety. Then, $\Aut(\ti{X})$ is an algebraic group of dimension $4$ whose connected component that contains the identity is the torus $T$, i.e., $\ti{X}$ has no roots. Moreover, the exact sequence (\ref{EQ:DEMAZURE}) becomes
\begin{equation}
\label{EQ:DEMAZURE2}
1\rightarrow T\rightarrow \Aut(X)\rightarrow \Aut(\Z^4,\ti{\Sigma})\rightarrow 1.
\end{equation}\end{lem}
\begin{proof}
It suffices to show that there are no roots. Standard facts of toric geometry imply that the number of roots of a simplicial and Gorenstein variety is given by the sum over the facet $\Gamma$ of $\Delta_{-K_{\ti{X}}}$ of $l^*(\Gamma)$, i.e., the number of integral points in $\Gamma$ that are not on its topological border. $\ti{\Sigma}$ is a maximal projective subdivision and this implies that the polytope associated to the divisor $-K_{\ti{X}}$ is the polytope of $\du{X}$. It is easy to see that the $16-$cells\footnote{Recall that $\du{\Delta}=\Ch(\{\pm \e_1,\pm \e_2,\pm \e_3,\pm \e_4\})$.} contain no integral points besides its vertices and the origin. This implies that there are no roots and that the connected component that contains the identity is given by the torus.
\vspace{2mm}


\noindent To prove the second statement, we need to show that for every pair $(\rho_1,\rho_2)$ of distinct indices of rays (for brevity, we will often write $\rho$ meaning a ray of a fan or its primitive generator $v_{\rho}$), $D_{\rho_1}$ is different from $D_{\rho_2}$ inside $\A_{3}(\ti{X})$. Assume that $(\rho_1,\rho_2)$ is such a pair. Then, by definition, $\deg{x_{\rho_1}}=\deg{x_{\rho_2}}$ and $(x_{\rho_1},x^{D_{\rho_2}})=(x_{\rho_1},x_{\rho_2})$ is a root. But we have shown that $\ti{X}$ has no roots so the partition given by $\sim$ is the trivial one. In particular, $\Pi_i S_i$ is the trivial group.
\end{proof}

\noindent We have shown that, by changing the projective subdivision, one can change the part of $\Aut(\ti{X})$ that depends on the automorphism of the fan. This depends strongly on how much the subdivision is invariant under the hyperoctahedral group $B_4$ that is the group of the automorphisms of the hypercube and of the fan spanned by its faces. In particular, $\Aut(\Z^4,\ti{\Sigma})$ will be a subgroup of $B_4$.
\vspace{2mm}

\noindent An example of maximal projective subdivision that preserves the full symmetry of the hypercube is the one that has, as maximal cones, the cones spanned by a flag in the hypercube. By a {\it flag}, we denote a tetrahedron with one vertex on the center of one of the $8$ cubes of $\Delta$ (call it $C$), another vertex on the center of a $2-$face of $C$ (denote the face by $F$), the third vertex on an edge $E$ of $F$ and the final vertex on one of the two end-point of $F$, i.e. on a vertex of the hypercube. The set of such flags cover the hypercube and two flags meets only along a lower dimensional face of the two tetrahedra. This implies that the set of the cone spanned by the flags generate a fan with $384$ facets (one for each flag). By construction, the set of the primitive generators of each maximal cone - which are the vertices of a flag in the hypercube - is a $\Z-$basis for $\Z^4$. This imply that the resolution is smooth and is a maximal subdivision of $\du{\Sigma}$. It is also easy to see 
that the corresponding toric variety is projective by constructing an explicit support function iteratively. We will call this subdivision the {\it flag subdivision} of $\du{X}$.
\begin{figure}[H]
\begin{center}
\includegraphics[width=.8\textwidth]{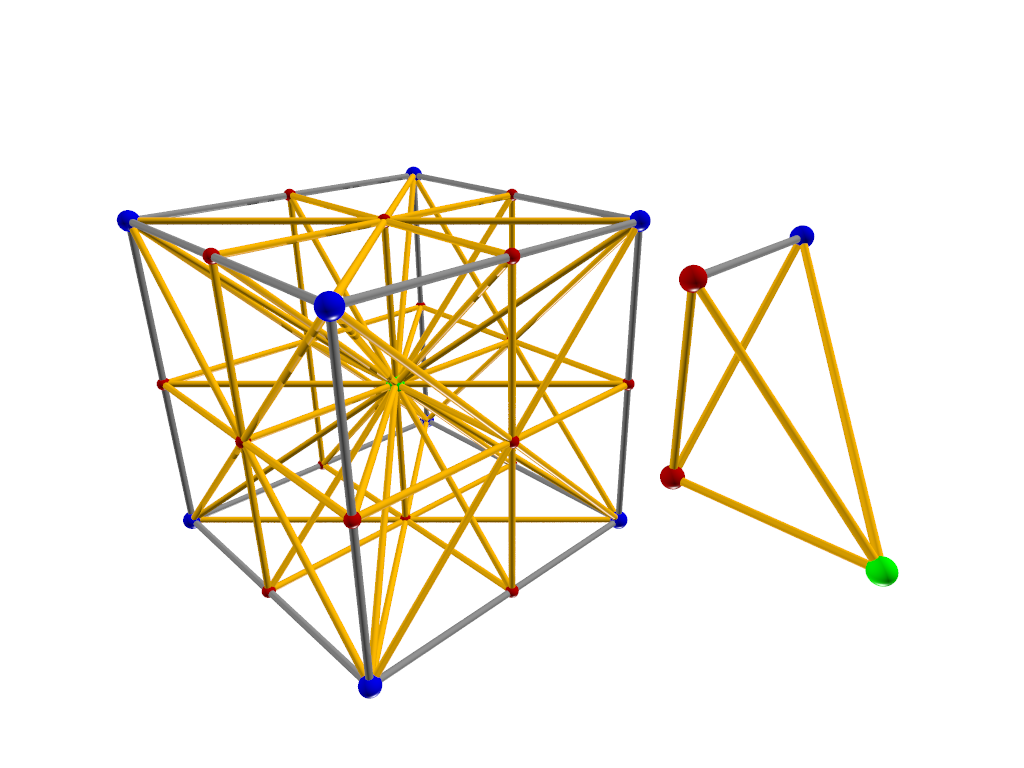}
\caption{The fan of the hypercube has $8$ maximal cones. In the picture we show the slice of the subdivision of the cone corresponding to the cube with center $(1,0,0,0)$ with the plane $x_1=1$. On the left, one of its $64$ tetrahedra (each tetrahedron corresponds to one of the maximal cones of the flag subdivision).}
\end{center}
\end{figure}

\section{Choosing the Group}
\label{SEC:CHOSING}
\noindent To compute the invariants of the action of a subgroup of $\Aut(\ti{X})$ on the cohomology of $\ti{X}$, it is sufficient to consider the action of its projection on $\Aut(\Z^4,\ti{\Sigma})$; in fact, $T$ acts trivially. Let $K$ be a subgroup of $B_4$ and assume that a $K-$invariant subdivision has been chosen. Call $\ti{Y}$ the generic element of $|-K_{\ti{X}}|$ that is a smooth Calabi-Yau that it is the mirror of $Y$. Assume moreover that $\ti{Y}$ is $\ti{G}-$invariant with $\ti{G}\leq \Aut(\ti{X})$ such that the image of $\ti{G}$ in $\Aut(\Z^4,\ti{\Sigma})$ is $K$. Since $\ti{X}$ is simplicial (i.e., it has only orbifold singularities), the Picard group has finite index in $\A_{3}(\ti{X})$. This implies that $\A_{3}(\ti{X})\otimes\Q=\Pic(\ti{X})\otimes\Q=\Pic_{\Q}(\ti{X})$. Consider the exact sequence 
\begin{equation}
\label{EQ:PIC}
\xymatrix{
0 \ar[r] & \Z^4 \ar[r]^-{\alpha} & \Z^{\ti{\Sigma}(1)}:=\bigoplus_{\rho\in \ti{\Sigma}} \Z \ar[r]^-{\beta} & \A_{3}(\ti{X}) \ar[r] & 0} 
\end{equation}
where $\alpha(m)=(<v_{\rho},m>)_{\rho}$ and $\beta((a_\rho)_{\rho})=\sum_{\rho}a_\rho D_{\rho}$. By tensorizing with $\Q$, we obtain the exact sequence
\begin{equation}
\label{EXSEQ:PICTENSOR}
\xymatrix{
0 \ar[r] & \Q^4 \ar[r]^-{\alpha} & \Q^{80} \ar[r]^-{\beta} & \Pic_{\Q}(\ti{X}) \ar[r] & 0},
\end{equation}
from which we deduce that $\Pic_{\Q}(\ti{X})$ has dimension $76$.
\vspace{2mm}

\noindent Since $\ti{Y} \subset \ti{X}$, there is a map from $\Pic_{\Q}(\ti{X})$ to $\Pic_{\Q}(\ti{Y})$. By standard facts in toric geometry (see, for instance, \cite{buc}), the kernel of this map is generated by the classes $D_{\rho}$ of divisors associated to edges of $\du{(\du{\Delta})}=\Delta$ spanned by integral vectors, which are in the interior of some facets of $\Delta$. In our case, $\Delta$ is the hypercube and the toric divisors in the kernel are the ones corresponding to the centers of the $8$ cubes of $\Delta$. Therefore, the dimension of the kernel in $8$. Consider the exact sequence
\begin{equation}
\label{EQ:exseq}
\xymatrix{
0 \ar[r] & \Ker \ar[r] & \Pic_{\Q}(\ti{X}) \ar[r] & \Pic_{\Q}(\ti{Y}) \ar[r] & 0}.
\end{equation}

\noindent If we take the invariant subspaces in (\ref{EQ:exseq}), we obtain the exact sequence (see, for instance, \cite{Wei94})
\begin{equation}
\label{EQ:exseqinv}
\xymatrix{
0 \ar[r] & \Ker^K \ar[r] & \Pic_{\Q}(\ti{X})^K \ar[r] & \Pic_{\Q}(\ti{Y})^K \ar[r] & 0} 
\end{equation}
that allows us to calculate the invariant part of $\Pic_{\Q}(\ti{Y})$.
Since $\ti{Y}$ is a smooth Calabi-Yau, we have $\H^2(\ti{Y},\C)=\H^{1,1}(\ti{Y})$. We then need to search for a group $K$ of order $16$ such that the dimension of $\Pic_{\Q}(\ti{Y})^K$ is $5$. 
\vspace{2mm}

\noindent Up to conjugacy, the group $B_4$ has exactly $37$ subgroups of order $16$.
By means of a computer search, we can compute the dimension of the vector spaces in the previous exact sequence. The following table summarizes these results. For each group we specify also the GAP Index\footnote{GAP stands for ``Groups, Algorithms, Programming" and is a software for computational discrete algebra. It has a database for finite groups of order less than or equal to $2000$ (except $1024$). For example, groups of order $16$ are identified by a pair $(16,x)$ where $x$ is a progressive number. The GAP Index of a group of order $16$ is $x$.} by which it is possible to identify the isomorphism class of the group from the GAP database. For example, the group relative to line $7$ of the table, has GAP index $11$ that corresponds to $\D_8\times \Z_2$. 
\begin{small}
\begin{center}
\begin{tabular}{c|c|cccc|c}
\label{TAB:G16}
 & GAP Index & $(\Q^{4})^{K}$ & $(\Q^{80})^{K}$ & $\Pic_{\Q}(\ti{X})^K$ & $\Ker^{K}$ & $\Pic_{\Q}(\ti{Y})^K$ \\ \hline
1  &  14  &  0  &  15  &  15  &  4  &  11 \\
2  &  14  &  0  &  15  &  15  &  2  &  13 \\
3  &  3  &  0  &  9  &  9  &  2  &  7 \\
4  &  11  &  0  &  10  &  10  &  2  &  8 \\
5  &  3  &  0  &  9  &  9  &  2  &  7 \\ \hline
6  &  2  &  0  &  8  &  8  &  2  &  6 \\
7  &  11  &  0  &  10  &  10  &  2  &  8 \\
8  &  3  &  0  &  9  &  9  &  1  &  8 \\
9  &  3  &  0  &  10  &  10  &  2  &  8 \\
10  &  11  &  0  &  12  &  12  &  2  &  10 \\ \hline
11  &  10  &  0  &  11  &  11  &  3  &  8 \\
12  &  3  &  0  &  9  &  9  &  2  &  7 \\
13  &  10  &  0  &  11  &  11  &  2  &  9 \\
14  &  13  &  0  &  8  &  8  &  1  &  7 \\
15  &  11  &  0  &  10  &  10  &  1  &  9 \\ \hline
16  &  11  &  0  &  11  &  11  &  3  &  8 \\
17  &  14  &  0  &  15  &  15  &  3  &  12 \\
{\bf 18}  &  {\bf 6}  &  {\bf 0}  &  {\bf 6}  &  {\bf 6}  &  {\bf 1}  &  {\bf 5} \\
19  &  11  &  0  &  12  &  12  &  3  &  9 \\
20  &  11  &  0  &  14  &  14  &  3  &  11 \\ \hline
21  &  3  &  0  &  10  &  10  &  2  &  8 \\
22  &  11  &  0  &  12  &  12  &  2  &  10 \\
23  &  4  &  0  &  8  &  8  &  2  &  6 \\
24  &  11  &  0  &  11  &  11  &  2  &  9 \\
25  &  11  &  0  &  12  &  12  &  3  &  9 \\ \hline
26  &  11  &  0  &  12  &  12  &  1  &  11 \\
27  &  8  &  0  &  7  &  7  &  1  &  6 \\
28  &  7  &  0  &  9  &  9  &  1  &  8 \\
29  &  3  &  0  &  10  &  10  &  1  &  9 \\
30  &  11  &  0  &  13  &  13  &  3  &  10 \\ \hline
31  &  11  &  0  &  11  &  11  &  3  &  8 \\
32  &  11  &  0  &  13  &  13  &  3  &  10 \\
33  &  11  &  1  &  17  &  16  &  4  &  12 \\
34  &  11  &  1  &  17  &  16  &  3  &  13 \\
35  &  11  &  0  &  13  &  13  &  2  &  11 \\ \hline
36  &  11  &  0  &  11  &  11  &  2  &  9 \\
37  &  11  &  0  &  13  &  13  &  2  &  11 
\end{tabular}
\end{center}
\end{small}

\noindent As it is possible to see from the table, there is only one group (up to conjugacy) such that $\h^{1,1}(\ti{Y})=5$.
This group has GAP index $6$ and it is isomorphic to a particular semidirect product of $\Z_8$ and $\Z_2$, which is denoted by $M_{16}$ with the following presentation
$$M_{16}=\{g,h \, |\, g^8=h^2=\id\quad h^{-1}gh=g^{5}\}.$$
By direct inspection, one checks that $B_4$ has exactly $6$ subgroups isomorphic to $M_{16}$ and they are all conjugated.
Then the following result holds.
\begin{prop}
Assume that $\ti{\Sigma}$ is a maximal projective subdivision of $\Sigma$ such that $\ti{X}=X_{\ti{\Sigma}}$ admits a smooth anticanonical section $\ti{Y}$ and a group $\ti{G}$ with $\h^{1,1}(\ti{Y}/\ti{G})=5$. Then, $\ti{\Sigma}$ must have a $M_{16}-$symmetry.
\end{prop}

\noindent In what follows, $\ti{\Sigma}$ is a maximal projective subdivision of $\du{\Sigma}$ that has a $M_{16}-$symmetry. Without loss of generality, we can assume $$\spn{g,h}:=:K\leq \Aut(\Z^4,\ti{\Sigma})$$ with
\begin{equation}
\label{EQ:GROPGENS}
g=\begin{bmatrix}
 0 &-1 & 0 & 0 \\
 0 & 0 & 0 &-1 \\
-1 & 0 & 0 & 0 \\
 0 & 0 & 1 & 0 \\
\end{bmatrix}\quad h=\begin{bmatrix}
 1 & 0 & 0 & 0 \\
 0 &-1 & 0 & 0 \\
 0 & 0 &-1 & 0 \\
 0 & 0 & 0 & 1 \\
\end{bmatrix}. 
\end{equation}
\noindent It will be useful to observe that $M_{16}$ has three involutions, namely $(g^4, g^4h, h)$, and that they generate a subgroup of order $4$isomorphic to $\Z_2\oplus\Z_2$. Moreover, among the $3$ involutions, only $g^4$ has square roots (that are $g^2$ and $g^6$). Recall that we need a group $\ti{G}$ of order $16$ such that its projection on $\Aut(\Z^4,\ti{\Sigma})$ is $K$.
Now we will describe all such groups $\ti{G}$.
\vspace{2mm}

\noindent As before, call $T$ the maximal torus inside $\ti{X}$. Recall that $T$ is an open dense subset of $\ti{X}$ and that we can write it as the toric affine patch corresponding to the cone $\{0\}$ od $\ti{\Sigma}$, i.e., it has a realization inside $\ti{X}$ given by 
$$X_{\{0\}}=\Spec(\C[\{0\}^\vee\cap \Z^4])=\Spec(\C[x_1^{\pm 1},x_2^{\pm 1},x_3^{\pm 1},x_4^{\pm 1}]).$$
Choose coordinates $(\lambda_i)$ on $T$ such that the action of the torus on $T=X_{\{0\}}$ is given by
$$\begin{array}{ccc}
T\times X_{\{0\}} & \longrightarrow & X_{\{0\}}\\
((\lambda_i)_i\,,(x_i)_i) & \mapsto & (\lambda_ix_i)_i.
\end{array}$$
In coordinates, for every $\lambda \in T$ one has 
\begin{footnotesize}
$$\lambda:\left(x_1,x_2,x_3,x_4,\frac{1}{x_1},\frac{1}{x_2},\frac{1}{x_3},\frac{1}{x_4}\right)\mapsto
\left(\lambda_1x_1,\lambda_2x_2,\lambda_3x_3,\lambda_4x_4,\frac{1}{\lambda_1x_1},
\frac{1}{\lambda_2x_2},\frac{1}{\lambda_3x_3},\frac{1}{\lambda_4x_4}\right).$$
\end{footnotesize}

The action of $K$ on the points of $X_{\{0\}}$ can be deduced from that on the coordinates. For example, consider $g\in K$ whose action on the lattice $\Z^4=N=M^\vee$ is described by the matrix in (\ref{EQ:GROPGENS}). The action of its dual $g^\vee:N^\vee\rightarrow N^\vee$ is described by the transpose of the matrix in (\ref{EQ:GROPGENS}): it is a permutation on the set $\{\pm \e_i^\vee\}_i$. Here, we denote by $\e_i$ the elements of the standard $\Z-$basis of $N$ and $\e_i^\vee$ its dual. Now, recall that the coordinates $x_i^{\pm 1}$ on the torus are, by definition, of the form $\chi^{\pm\e_i^\vee}$. This allows us to deduce the action of $g$ (as of any other element of $\Aut(\Z^4,\ti{\Sigma})$) on the torus.
Thus the description of the action of $g$ and $h$ on $T$ (which will be denoted again by $g$ and $h$) are
\begin{footnotesize}
$$g:\left(x_1,x_2,x_3,x_4,\frac{1}{x_1},\frac{1}{x_2},\frac{1}{x_3},\frac{1}{x_4}\right)\mapsto
\left(\frac{1}{x_2},\frac{1}{x_4},\frac{1}{x_1},x_3,x_2,x_4,x_1,\frac{1}{x_3}\right)$$
$$h:\left(x_1,x_2,x_3,x_4,\frac{1}{x_1},\frac{1}{x_2},\frac{1}{x_3},\frac{1}{x_4}\right)\mapsto
\left(x_1,\frac{1}{x_2},\frac{1}{x_3},x_4,\frac{1}{x_1},x_2,x_3,\frac{1}{x_4}\right)$$
\end{footnotesize}
Two generators of the group that projects on $K$ are of the form $\{\lambda\circ h,\mu\circ g\}$, where $\lambda,\mu\in T$.
By direct inspection, one can check that $\spn{\lambda\circ h,\mu\circ g}$ is isomorphic to $M_{16}$ if and only if the following conditions are satisfied:
\begin{equation}
\left\lbrace
\begin{array}{l}
\lambda_1^2=\lambda_4^2=1, \\
\lambda_1\lambda_3\mu_1\mu_4/(\mu_2\mu_3)=1, \\
\mu_2\mu_3/(\lambda_3\lambda_4\mu_1\mu_4)=1, \\
\lambda_1\mu_1\mu_2/(\lambda_2\mu_3\mu_4)=1, \\
\lambda_2\lambda_4\mu_3\mu_4/(\mu_1\mu_2)=1,
\end{array}
\right.
\end{equation}
that is to say,
\begin{equation}
\left\lbrace
\begin{array}{l}
\lambda_1^2=1, \\
\lambda_2=\lambda_1\mu_1\mu_2/(\mu_3\mu_4), \\
\lambda_3=\mu_2\mu_3/(\lambda_1\mu_1\mu_4), \\
\lambda_4=\lambda_1.
\end{array}
\right.
\end{equation}
Thus, we have $2$ families (one for each choice of $\lambda_1\in\{\pm 1\}$) of groups the satisfies our requirements.
If one takes $\nu\in T$ to be $(m_1^{-1}c_2^{-1},c_2,m_4m_2^{-1}c_2^{-1},m_2^{-1}c_2^{-1})$ with $c^2=m_1^{-1}m_2^{-1}m_3m_4$ then
$$\nu^{-1}\circ (\lambda\circ h)\circ \nu = (\lambda_1,\lambda_1,\lambda_1^{-1},\lambda_1^{-1})\circ h,$$
$$\nu^{-1}\circ (\mu\circ g)\circ \nu = g,$$
with $\lambda_1^2=1$ so, up to conjugacy, there are only two subgroups of $\Aut(\ti{X})$ that project themselves on $K$:
$$L_1:=\spn{g,h}\quad \mbox{ and }\quad L_2:=\spn{g,(-1,-1,-1,-1)\circ h}.$$
Define $h_1:=h$ and $h_2:=(-1,-1,-1,-1)\circ h$, so that $L_i=\spn{g,h_i}$.

\section{Anticanonical $L_i-$Invariant Sections and Fixed Loci}
\label{SEC:ANTI}

\noindent In this section, we analyse invariant sections and their relation with the fixed locus of $L_1$ and that of $L_2$. For most of the section, the toric variety is a generic maximal projective resolution $\ti{X}$ that has the $M_{16}-$symmetry described before. For more details about fixed points we will, however, restrict to consider the flag resolution defined at the end of Section (\ref{SEC:AUTORES}).
\vspace{2mm}

\noindent First, we will analyse the invariant space of anticanonical sections of $L_1$ and $L_2$.
By standard toric geometry, $\H^0(\ti{X},-K_{\ti{X}})$ has dimension equal to $|\Z^4\cap \du{\Delta}|$. Moreover, from this set it is possible to write down explicitely a basis for the space in the coordinates of the torus $X_{\{0\}}$ (which we will identify with $T$). More precisely, given a divisor $D$ in a projective toric variety $X$ constructed from a polytope in a lattice $M$, there is an isomorphism between $\H^0(X,D)$ and the set 
$$L(D):=\bigoplus_{m\in \Delta_{D}\cap M}x^m\cdot \C,$$
where $\Delta_D$ is the polytope associated to the divisor $D$, $x_i$ are the coordinates of the torus and $x^m$ is a shorthand for $\Pi_{i}x_i^{m_i}$.
\vspace{2mm}

\noindent In our case, because $\ti{\Sigma}$ is a maximal projective subdivision, one has $\Delta_{-K_{\ti{X}}}=\du{\Delta}$, the $16-$cells. The integral points of $\du{\Delta}$, i.e. the points of the set $\du{\Delta}\cap \Z^4$, are the coordinate points, their opposite and the origin. Hence there are $9$ indipendent anticanonical sections and 
$$\H^0(\ti{X},-K_{\ti{X}})\simeq L(\du{\Delta})=\left\lbrace\sum_{i}a_ix_i+\sum_{i}b_i\frac{1}{x_i}+c\right\rbrace.$$
From this description, it is also easy to describe the invariant subspaces with respect to the actions of $L_1$ and $L_2$.
\vspace{2mm}

\noindent We will write $V_{a,b}^{(i)}$ to denote a subspace of $\H^0(\ti{X},-K_{\ti{X}})$ such that $g(v)=av$ and $h_i(v)=bv$ for every $v\in V_{a,b}^{(i)}$. By direct computation, we obtain
$$\H^0(\ti{X},-K_{\ti{X}})=V_{1,1}^{(1)}\oplus V_{-1,1}^{(1)}\oplus V_{i,1}^{(1)}\oplus V_{-i,1}^{(1)}\oplus W^{(1)},$$
$$\H^0(\ti{X},-K_{\ti{X}})=V_{1,1}^{(2)}\oplus V_{1,-1}^{(2)}\oplus V_{-1,-1}^{(2)}\oplus V_{i,-1}^{(2)}\oplus V_{-i,-1}^{(2)}\oplus W^{(2)},$$
where $W^{(i)}$ is a subspace that doesn't contains any invariant subspace of dimension $1$. Apart from $V_{1,1}^{(1)}$ which has dimension $2$, every other subspace of the form $V_{a,b}^{(i)}$ has dimension $1$.

\subsection{The group $L_1$}

The invariant subspace with respect of $L_1$ has dimension $2$ and it is generated by
$$I_0:=x_1+\frac{1}{x_1}+x_2+\frac{1}{x_2}+x_3+\frac{1}{x_3}+x_4+\frac{1}{x_4}\quad \mbox{ and }\quad I_1:=1.$$
The generic invariant section is then
\begin{multline}
s:=\frac{1}{x_1x_2x_3x_4}\left(ax_1x_2x_3x_4+b(x_1^2x_2x_3x_4+x_2x_3x_4+x_1x_2^2x_3x_4+x_1x_3x_4+\right.\\
\left. +x_1x_2x_3^2x_4+x_1x_2x_4+x_1x_2x_3x_4^2+x_1x_2x_3)\right) 
\end{multline}
whose zero locus can be proven to be smooth in $T$ for generic values of $a$ and $b$. Its closure inside $\ti{X}$ is an invariant anticanonical section. Let's focus on the fixed locus of $L_1$ on $T$. Denote by $V_4$ the subgroup of $L_1$ generated by $g^4$ and $h_1$. As observed before, it is isomorphic to $\Z_2\oplus\Z_2$ and consists of all the involutions of $L_1$ plus the identity. For this reason, the fixed locus of $L_1$ is that of $V_4$. In $K$, $g^4$ equals to $-\id$, so $g^4$ acts on $T$ by sending $x_i$ to $x_i^{-1}$. The fixed locus of $g^4$ inside $T$ consists of those points on the torus with coordinates $x_i$ such that
$$\left(x_1,x_2,x_3,x_4\right)=\left(\frac{1}{x_1},\frac{1}{x_2},\frac{1}{x_3},\frac{1}{x_4}\right),$$
i.e., the $16$ points with $x_i=\pm 1$. By direct inspection, the generic section $s$ doesn't vanish on any of these points.
\vspace{2mm}

\noindent The fixed points of $h_1$ are those for which
$$\left(x_1,x_2,x_3,x_4\right)=\left(x_1,\frac{1}{x_2},\frac{1}{x_3},x_4\right),$$
so they have $x_2,x_3\in \{\pm1\}$ and $x_1,x_4$ free. Every section meets each of the four component of the fixed locus in a curve that is smooth on $T$. A similar result is true for $g^4h_1$.
\vspace{2mm}

\noindent This shows that the generic $L_1-$invariant section meets the fixed locus of $L_1$ in $8$ curves and is smooth in $T$. The elements of $L_1$ that have fixed points (on $T$) on $V(s)$ are only $h$ and $g^4h$ because $g^4$ is the only involution that has a square root and it acts freely. It remains to consider what happens outside $T$, i.e., on the toric divisors of $\ti{X}$. 

\begin{lem}
\label{LEM:FIXPINT}
Let $P$ be a fixed point of $g$ in $\ti{X}$. Then $P$ lies in $T$.
\end{lem}
\begin{proof}
If $P$ is a fixed point of $g$, then it is also a fixed point for $g^4$. The toric variety $\ti{X}=X_{\ti{\Sigma}}$ is covered by affine patches $X_{\sigma}$ where $\sigma$ is a facet in $\ti{\Sigma}$. In particular, there exists $\sigma$ such that $P\in X_{\sigma}$. Call $\tau$ the cone that is the image of $\sigma$ by the involution $g^4$. Being a fixed point, $P$ will also be a point of $X_{\tau}$ because $g^4(X_{\sigma})=X_{\tau}$. We have then $$P\in X_{\sigma}\cap g^4(X_{\sigma})=X_{\sigma}\cap X_{\tau}=X_{\sigma\cap\tau}.$$
But $g^4$ acts on $\Z^4$ as $-\id$ so $\tau=g^4(\sigma)=-\sigma$. Being $\sigma$ a strictly convex rational cone one has then
$\sigma\cap(-\sigma)=\{0\}$. In the end, hence $P\in X_{\sigma\cap\tau}=X_{\{0\}}=T$.
\end{proof}

\begin{coroll}
For every maximal projective subdivision $\ti{X}$ of $\du{X}$ that admits a $K-$symmetry, the automorphism $g$ acts freely on the whole generic invariant anticanonical section.
\end{coroll}

\noindent Assume now that $\ti{X}$ is the flag subdivision. Let's prove that the closure of the generic $V(s)$ is smooth in the affine patch $X_{\sigma}$ where $\sigma$ is the facet generated by the flag whose vertices are 
$$(1,0,0,0),(1,1,0,0),(1,1,1,0),(1,1,1,1).$$
The dual cone is generated by the following primitive vectors
$$(1,-1,0,0),(0,1,-1,0),(0,0,1,-1),(0,0,0,1)$$
so that the $\C-$algebra generated by the semigroup $\sigma^\vee\cap \Z^4$ is 
$$\C[x_1x_2^{-1},x_2x_3^{-1},x_3x_4^{-1},x_4]=\C[X_1,X_2,X_3,X_4].$$
The affine coordinates on $X_{\sigma}$ are then $\{X_1,X_2,X_3,X_4\}$.
The relations that allow to go from the $X_i$ to the coordinates on the torus are
$$\left\lbrace
\begin{array}{cl}
x_1= & X_1X_2X_3X_4\\
x_2= & X_2X_3X_4\\
x_3= & X_3X_4\\
x_4= & X_4
\end{array}
\right.$$
Making the substitution one has 
\begin{multline}
s\cdot x_1=:S=aX_1X_2X_3X_4 +b(X_1^2X_2^2X_3^2X_4^2 + X_1X_2^2X_3^2X_4^2 + \\
 + X_1X_2X_3^2X_4^2+ X_1X_2X_3X_4^2 + X_1X_2X_3 + X_1X_2 + X_1 + 1).      
\end{multline}
We multiply by $x_1$ because, in this way, one doesn't change the zero locus in $T$ (in fact $x_1$ is invertible in the coordinate ring of $T$) but allows, with the substitution, to cancel other components that are not in the closure.
By direct inspection one sees that $S$ is (for generic $a,b$) smooth. This shows that there exist a smooth Calabi-Yau threefold $\ti{Y}$ that is $L_1-$invariant (and that is also a mirror for $Y$). 
\vspace{2mm}

\noindent As a consequence of a result\footnote{If $X$ is a Calabi-Yau threefold and $\iota$ is a small involution acting on $X$ such that its fixed locus contains a curve, then $\iota$ is a symplectic automorphism of $X$} in \cite{Fav13}, we deduce that $L_1$ is a symplectic subgroup of $\Aut(\ti{Y})$. In fact $g$ acts freely on $\ti{Y}$ (by direct check inside the torus and using the Lemma for the outside) and $h$ is an involution that fixes curves. With this we are able to complete the proof of the following theorem.

\begin{thm}
\label{THM:MIRROR}
There exists a maximal projective subdivision $\ti{\Sigma}$ (i.e. the flag subdivision) of $\du{\Sigma}$ and its associated toric variety $\ti{X}$, a smooth Calabi-Yau threefold $\ti{Y}$ in $\ti{X}$ and a group $\ti{G}\leq \Aut(\ti{X})$ (namely $L_1$) such that:
\begin{itemize}
\item $\ti{Y}$ is $\ti{G}-$invariant and a mirror for $Y\in |-K_X|$;
\item $\ti{G}$ is symplectic as subgroup of $\Aut(\ti{Y})$ and acts with a fixed locus with irreducible components that are curves;
\item $Z:=\ti{Y}/\ti{G}$ is a Calabi-Yau orbifold whose singular locus has pure dimension $1$ and it has fundamental group isomorphic to $\Z_4$;
\item $Z$ is a topological mirror for $Y/G$, i.e. it has $(\h^{1,1}(Z),\h^{1,2}(Z))=(5,1)$.
\end{itemize}
\end{thm}

\begin{proof}
The only thing left is that $Z$ is a topological mirror for $Y/G$, i.e. we have to prove that $\h^{1,2}(Z)=\h^{1,2}(\ti{Y})^{\ti{G}}=1$ and that its fundamental group is $\Z_4$. 
\vspace{2mm}

\noindent Following Batyrev's work we have that the space $\H^{2,1}(\ti{Y})=\H^{1}(\ti{Y},T_{\ti{Y}})$ classifies the infinitesimal deformation of $\ti{Y}$. It has a subspace
$$\H^{2,1}_{poly}(\ti{Y})\leq \H^{2,1}(\ti{Y}),$$
which parametrizes the deformation determined by the hypersurfaces $\ti{Y}\in |-K_{\ti{X}}|$. The other, possible, deformations of $\ti{Y}$ form a subspace of dimension
$$D:=\sum_{\du{\Theta}}l^*(\du{\Theta})l^*(\du{\hat{\Theta}}),$$
where $\du{\Theta}$ runs over all the faces of $\du{\Delta}$ of codimension $2$ and $\du{\hat{\Theta}}$ is the face of $\Delta$ dual to $\du{\Theta}$. But $\du{\Delta}$ is the $16-$cells and each face has no points in the relative interior so $D=0$. This shows that all the contribution to $\H^{2,1}(\ti{Y})$ cames from $\H^{2,1}_{poly}(\ti{Y})$.
\vspace{2mm}

\noindent Hence, it suffices to analyse the action of $L_1$ on $\H^{0}(\ti{X},-K_{\ti{X}})$. But, as we have seen before, 
$$\H^0(\ti{X},-K_{\ti{X}})\simeq L(\du{\Delta})$$
and the action of $L_1$ on this space is determined by the permutation of the points in $\du{\Delta}$. We have already calulated the invariant subspace and it has dimension $2$. This implies that $\H^{2,1}(\ti{Y})^{L_1}=2-1=1$.
\vspace{2mm}

\noindent In (\cite{BK06}), it has been calculated the fundamental group of the generic anticanonical section of a Fano fourfold. Only a few of the varieties analysed have non-trivial fundalmental group and the case which concern us now, the anticanonical section in the dual of $(\P^1)^4$, it is not one of those. This shows that $\du{Y}$ is simply connected and thus the same holds for $\ti{Y}$. By the main result of (\cite{Arm82}), the fundamental group of $\ti{Y}$ is isomorphic to the quotient of $\ti{G}$ by the group generated by the elements whose fixed locus is not empty. As we have seen the only automorphisms with fixed points on $\ti{Y}$ are $h_1$ and $g^4h_1$ so the fundamental group of $Z$ is $\ti{G}/\spn{h_1,g^4h_1}\simeq \Z_4$.
\end{proof}

\subsection{The Group $L_2$}

In this case, the decomposition of $\H^0(\ti{X},-K_{\ti{X}})$ doesn't have invariant spaces of dimension greater than one that are intersections of eigenspaces of elements of $L_2$. So, to each invariant subspace of dimension $1$ (there are $5$ of those) we associate an anticanonical section that is $L_2-$invariant. We will see in a moment that none of these is a mirror for $Y$: they are all singular.
\vspace{2mm}

\noindent Consider the subspace $V_{1,1}^{(2)}$. This is spanned by the section
$$s_{1,1}=1$$
so its zero locus on the torus is empty. The zero locus of the corresponding section on $\ti{X}$ is the union of the toric divisors; Hence it is not only highly singular, but also reducible.
\vspace{2mm}

\noindent The other invariant subspaces are of the form $V_{i^a,-1}^{(2)}$ for $a=0,\dots,3$. If we denote by $s_{i^a,-1}$ an element that spans $V_{i^a,-1}^{(2)}$, we can consider the following sections:
$$
s_{ 1,-1}:=\left(x_1+\frac{1}{x_1}\right)+\left(x_2+\frac{1}{x_2}\right)+\left(x_3+\frac{1}{x_3}\right)+\left(x_4+\frac{1}{x_4}\right)
$$
$$
s_{-1,-1}:=\left(x_1+\frac{1}{x_1}\right)-\left(x_2+\frac{1}{x_2}\right)-\left(x_3+\frac{1}{x_3}\right)+\left(x_4+\frac{1}{x_4}\right)
$$
$$
s_{ I,-1}:=\left(x_1+\frac{1}{x_1}\right)+i\left(x_2+\frac{1}{x_2}\right)-i\left(x_3+\frac{1}{x_3}\right)-\left(x_4+\frac{1}{x_4}\right)
$$
$$
s_{-I,-1}:=\left(x_1+\frac{1}{x_1}\right)-i\left(x_2+\frac{1}{x_2}\right)+i\left(x_3+\frac{1}{x_3}\right)-\left(x_4+\frac{1}{x_4}\right)
$$
One can show that the points of the torus with coordinates
$$(-1,-1,1,1),(1,1,-1,-1),(-1,1,-1,1),$$
$$(1,-1,1,-1),(1,-1,-1,1),(-1,1,1,-1)$$
are the singular points of $V(s_{1,-1})$ whereas
$$(1,1,1,1),(-1,-1,-1,-1),(-1,1,1,-1),(1,-1,-1,1)$$
are the singular points of $V(s_{i^a,-1})$ for $a=1,2,3$.
\vspace{2mm}

\noindent This concludes the case of $L_2$. We have also proved a refined version of Theorem (\ref{THM:MIRROR}): indeed,  we have also shown that if the pair $(\ti{Y},\ti{G})$ exists for some $\ti{X}$, as in the hypotesis of Theorem \ref{THM:MIRROR}, then $\ti{G}$ is uniquely determined by its conjugacy class.

\section{Fixed Locus of $L_1$ for the Flag Resolution}

By Theorem (\ref{THM:MIRROR}), we have seen that there exists a topological mirror for $Y/G$ and it is constructed as $Z=\ti{Y}/\ti{G}$ for suitable $\ti{Y}$ and $\ti{G}$. Furthermore, $Z$ is singular and all its deformations have at least the same singularities as $Z$: we have an equi-singular family of anticanonical quotients that is the (topological) mirror-family for $Y/G$. There is no hope for a smoothing to obtain a smooth mirror for $Y/G$ from $Z$. One can instead consider a crepant resolution. This exists because $\ti{Y}$ is a smooth Calabi-Yau threefold and $\ti{G}$ is a symplectic group. The crepant resolution is again a Calabi-Yau, it is smooth but it will have different Hodge numbers that are given by the dimension of the orbifold cohomology of $\ti{Y}/\ti{G}$. Now, we will analyse the fixed locus of the group $L_1$ action on the toric resolution associated to the flag subdivision and compute the orbifold cohomology for the quotient.
\vspace{2mm}

\noindent We recall that the fixed locus of $L_1$ on the torus $T$ is given by the union of the fixed loci of $h_1$ and $g^4h_1$. These two elements are the only ones that have non-empty fixed locus (on $\ti{Y}$). Both are symplectic involutions of $\ti{Y}$ and have fixed locus smooth of pure dimension $1$. If a curve of $\Fix(L_1)$ doesn't intersect $T$ then it is contained in $$\ti{Y}\cap \bigcup_{\rho} D_{\rho},$$
i.e., in the union of the toric divisors (restricted on $\ti{Y}$). Now, we will show that there are no more curves other than those we have found in $T$.
\vspace{2mm}

\noindent Assume that $P$ is a fixed point of $h_1$. The same methods apply also to $g^4h_1$. Call $\sigma$ a maximal cone of $\ti{\Sigma}$ such that $P\in X_{\Sigma}$. As we have done in Lemma (\ref{LEM:FIXPINT}), if $\tau=h(\sigma)$ we have 
$$P\in X_{\sigma}\cap X_{\tau}=X_{\sigma\cap\tau}.$$
In this case, however, it is indeed possible that $\sigma\cap \tau \neq \{0\}$ and these are the cases we have to investigate.
All the maximal cones of the flag subdivision have primitive generators for their rays of the form
$$\{\pm\e_i,\pm\e_i\pm\e_j,\pm\e_i\pm\e_j\pm\e_k,\pm\e_i\pm\e_j\pm\e_k\pm\e_l\}$$
where $\{i,j,k,l\}=\{1,2,3,4\}$. To have $\sigma\cap \tau\neq \{0\}$ it is necessary that $\pm \e_i$ is fixed by $h_1$. This happens if and only if $i\in\{1,4\}$. If that is the case, we have two further cases depending on whether $(i,j)$ is in 
$\{(1,4),(4,1)\}$ or not. If $(i,j)\in \{(1,4),(4,1)\}$ then $\sigma\cap \tau$ is the cone generated by $\pm\e_i$ and $\pm\e_i\pm\e_j$. In the other case $\sigma\cap \tau$ is a ray and has, as primitive generator, the vector $\pm\e_i$.
\begin{figure}[H]
\begin{center}
\includegraphics[width=.8\textwidth]{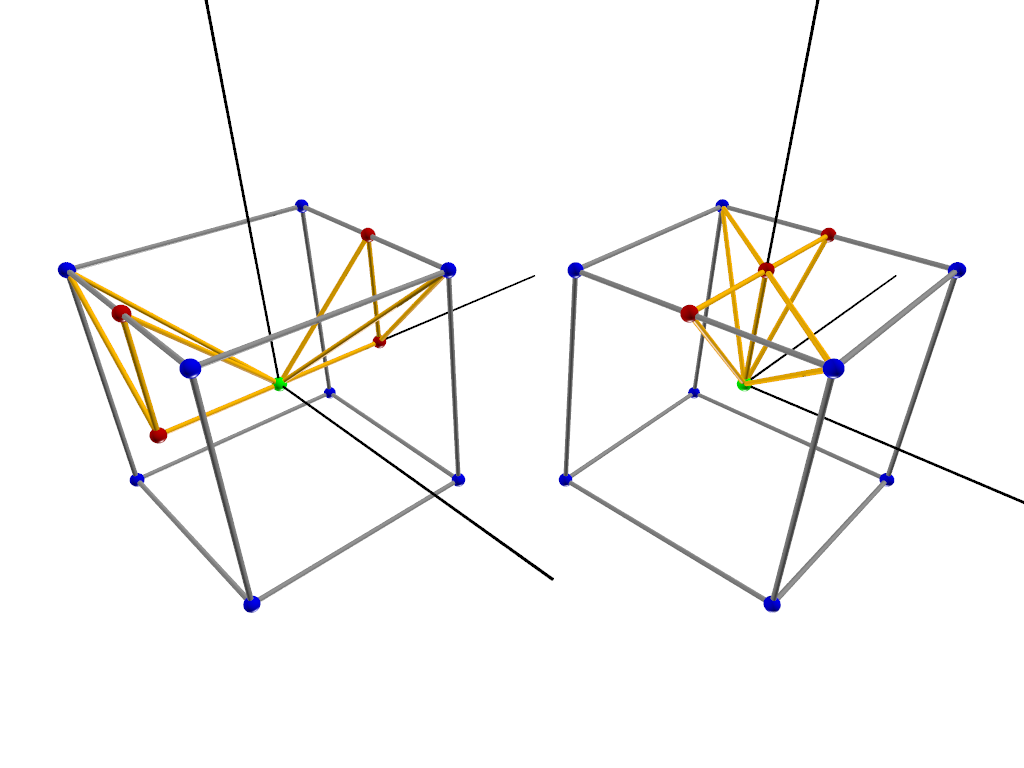}
\caption{The cube with center $(1,0,0,0)$, a tetrahedron and its image with respect to $h_1$ in two cases: the right one when $\sigma\cap h(\sigma) = \{0\}$ and the second when $\sigma\cap h(\sigma)=C(\e_1)$.}
\end{center}
\end{figure}

\noindent Assume that $\sigma\cap \tau = C(\e_1)$.
In this case $C(\e_1)^\vee = C(\e_1^\vee,\pm \e_2^\vee,\pm \e_3^\vee,\pm \e_4^\vee)$ and
$$\C[C(\e_1)^\vee\cap \Z^4]=\C[x_1,x_2^{\pm 1},x_3^{\pm 1},x_4^{\pm 1}].$$
The conditions to be satisfied by the coordinates of a fixed point of $h_1$ inside $X_{C(\e_1)}$ are $x_2,x_3=\pm1$.
If we define $\epsilon_2=\pm1$ and $\epsilon_3=\pm1$ such that $x_i=\epsilon_i$ for $i=2,3$, the equations for the fixed locus of $h_1$ on $\ti{Y}$ (restricted to $X_{C(\e_1)}$) are 
$$a\epsilon_2\epsilon_3x_1x_4+b(\epsilon_2\epsilon_3(x_1^2x_4+x_1x_4^2+x_1+x_4)+2(\epsilon_2+\epsilon_3)x_1x_4)=0.$$
For each choice of $\epsilon_2,\epsilon_3$ we have, generically, a smooth and irreducible curve that is not contained in the locus $x_1=0$. Our interest in this fact is justified because the points of $X_{C(\e_1)}$ with $x_1\neq 0$ are exactly the points of the torus. If we want to find new components for $\Fix(L_1)$, we have to search in the locus $x_1=0$. We have then show that there are no other components of $\Fix(L_1)$ in $X_{C(\e_1)}$. The projection of $L_1$ on $\Aut(\Z_4,\ti{\Sigma})$ acts transitively on ${\pm \e_i}$ so this shows that there are no more components of the fixed locus on 
$$X_{C(\e_1)}\cup X_{C(-\e_1)}\cup X_{C(\e_4)}\cup X_{C(-\e_4)}.$$

\noindent It remains to consider the case $(i,j)\in \{(1,4),(4,1)\}$. In total there are $8$ cones in $\ti{\Sigma}$ with these properties and they are arranged in $2$ orbits under the action of $M_{16}$. The orbits representatives are
$$C(\e_1,\e_1+\e_4)\quad \mbox{ and } \quad C(\e_1,\e_1-\e_4).$$
Assume that $\sigma\cap \tau = C(\e_1,\e_1+\e_4)$. The dual cone is $C(\e_1^\vee,\pm\e_2^\vee,\pm\e_3^\vee,-\e_1^\vee+\e_4^\vee)$
so that $$\C[C(\e_1,\e_1+\e_4)^\vee\cap \Z^4]=\C[x_1,x_2^{\pm 1},x_3^{\pm 1},x_4x_1^{-1}]:=\C[w_1,w_2^{\pm 1},w_3^{\pm 1},w_4].$$
In this affine patch, $\ti{Y}$ can be described as the zero locus of
\begin{multline*}
s:=aw_1w_2w_3w_4+b(w_1^2w_2w_3w_4^2 + w_1^2w_2w_3w_4 + w_1w_2^2w_3w_4 + \\
w_1w_2w_3^2w_4 + w_1w_2w_4 + w_1w_3w_4 + w_2w_3w_4 + w_2w_3).
\end{multline*}
The action of $h_1$ on these coordinates is given by
$$h:\left(w_1,w_2,w_3,w_4\right)\mapsto
\left(w_1,\frac{1}{w_2},\frac{1}{w_3},w_4\right)$$
so the fixed locus is given by $w_2,w_3\in \{\pm1\}$. These conditions give $4$ irreducible curves that are the closure of the ones we have found in the torus. The same argument applies when one search for new components in $X_{C(\e_1,\e_1-\e_4)}$.
\vspace{2mm}

\noindent In conclusion, $h_1$ has a fixed locus composed of $4$ irreducible curves and none of them is contained in a toric divisor. The only other element that has fixed points in $L_1$ is $g^4h_1$ but this is conjugated in $L_1$ to $h_1$. This means that none of the components of the fixed locus of $L_1$ is containtd in the intersection of $\ti{Y}$ and a toric divisor. It is also interesting to note that, by the main result of (\cite{Arm82}), the fundamental group of $\ti{Y}$ is $G/\spn{h_1,g^4h_1}\simeq \Z_4$.

\begin{thm}
\label{THM:RESOLUTION}
A crepant resolution of $Z$ is a smooth Calabi-Yau manifold with Hodge numbers $(\h^{1,1},\h^{1,2})=(8,4)$. 
\end{thm}
\begin{proof}
The formula for the orbifold cohomology of $Z=\ti{Y}/\ti{G}$ is the following:
$$\H_{orb}^{p,q}(Z)=\H^{p,q}(\ti{Y})^{\ti{G}}\oplus \bigoplus_{k\in CR_{\ti{G}}\setminus\{\id\}}\H^{p-\age(k),q-\age(k)}(\Fix(k)/C(k)),$$
where $CR_{\ti{G}}$ is a set whose elements are representant for the conjugacy classes of $\ti{G}$, $C(k)$ is the centralizer of $k$ in $\ti{G}$. In our particular case we can simplify the formula a lot because the only elements with fixed points are $h_1$ and $g^4h_1$ and they are conjugated. Being $h_1$ a symplectic involution one has that its age is $1$ so the formula becomes  
$$\H_{orb}^{p,q}(Z)=\H^{p,q}(\ti{Y})^{\ti{G}}\oplus \H^{p-1,q-1}(\Fix(h_1)/C(h_1)).$$
The shifts on the degree mean that we have a contribution to $\h^{1,1}=\h^{2,2}$ equal to the number of the components of $\Fix(h_1)/C(h_1)$ and a contribution to $\h^{1,2}=\h^{2,1}$ equal to the sum of the genus of the components of $\Fix(h_1)/C(h_1)$.
\vspace{2mm}

\noindent The centralizer of $h_1$ is the group generated by $g^2$ and $h$ that is abelian and isomorphic to $\Z_4\oplus\Z_2$. The action of $C(h_1)$ is not faithful; in fact, $h_1$ acts trivially. This means that $\Fix(h_1)/C(h_1)=\Fix(h_1)/\spn{g^2}$.
\vspace{6mm}

\noindent If, as before, $s$ is the section corresponding to $\ti{Y}$, denote 
$$C_{\pm\pm}:=V(s,x_2\mp1,x_3\mp1)$$
the four components of $\Fix(L_1)\cap T$. By definition it is clear that they are pairwise disjoint. We need to understand how $g^2$ acts on them. The action of $g^2$ on the coordinates is given by
$$g^2:\left(x_1,x_2,x_3,x_4\right)\mapsto
\left(\frac{1}{x_4},x_3,\frac{1}{x_2},x_1\right)$$
so $g^2(C_{++})=C_{++}, g^2(C_{--})=C_{--}$ and $g^2(C_{+-})=C_{-+}$. It is also of interest to note that $g^4(C_{\pm\mp})=C_{\pm\mp}$.
This means that
$$\frac{\Fix(L_1)}{\spn{g^2}}=\frac{C_{++}}{\spn{g^2}}\sqcup\frac{C_{--}}{\spn{g^2}}\sqcup\frac{C_{+-}}{\spn{g^4}}=
\frac{C_{++}}{\Z_4}\sqcup\frac{C_{--}}{\Z_4}\sqcup\frac{C_{+-}}{\Z_2}$$
and each of these quotients is free (because $g^4$, and hence $g^2$, doesn't have fixed points).
By the Riemann-Hurwitz formula we have
$$g(C_{++}/\Z_4)=\frac{g(C_{++})-1}{4}+1,\quad
g(C_{--}/\Z_4)=\frac{g(C_{--})-1}{4}+1$$
and
$$g(C_{+-}/\Z_2)=\frac{g(C_{+-})-1}{2}+1.$$
This, in particular, implies that $g(C_{++})$ and $g(C_{--})$ are congruent to $1$ modulo $4$ and that $g(C_{+-})=g(C_{-+})$ is odd.
The map $\pi:\ti{Y}\rightarrow \ti{Y}/\ti{G}$ is finite and it is ramified along $8$ curves. These are $C_{++},C_{--},C_{-+},C_{+-}$ (the fixed locus of $h_1$) and the images of these curves by $g$ (the fixed locus of $g^4h_1$). The two fixed locus are interchenged by $g$ so the branch locus is given by $\Fix(h_1)/C(h_1)$. We have already seen that is composed of $3$ curves. By writing explicitly the Riemann-Hurwitz formula, we obtain the following relation:
$$g(C_{++})+g(C_{--})+g(C_{+-})+g(C_{-+})=4.$$
This, and the arithmetic relations satisfied by $g(C_{\pm\pm})$, allow us to conclude that each component of the fixed locus and of the branch locus is an elliptic curve.
In conclusion we have
$$\h_{orb}^{1,1}(Z)=5+3\quad \mbox{ and }\quad \h_{orb}^{1,2}(Z)=1+3.$$
To conclude it remains to recall that the orbifold cohomology of a global quotient of a Calabi-Yau threefold by a finite symplectic group is the cohomology of one of its crepant resolution (cfr. \cite{Yas04}).
\end{proof}

\begin{remark}
\label{osette}
In \cite{BF11}, the authors studied a particular Calabi-Yau threefold. We will briefly recall its construction.
First of all, denote by $S$ the complete intersection in $\P^4$ of the following two quadrics:
$$f=x_0^2+x_1^2+x_2^2+x_3^2+x_4^2\quad\mbox{ and }g=x_0^2-ix_1^2-x_2^2+ix_3^2.$$
By simple calculation, $S$ is proved to be smooth and a del Pezzo surface of degree $4$.
Let $r$ and $t$ be the automorphisms of $\P^4\times \P^4$ such that 
$$\xymatrix@R=6pt{
(x,y) \ar@{|->}[r]^-r &  \left((x_0: x_1: -x_2: x_3: -x_4), (y_0: y_1: -y_2: y_3: -y_4)\right) \\
(x,y) \ar@{|->}[r]^-t &  \left((y_0: y_1: -y_2: -y_3: y_4), (x_0: x_1: x_2: x_3: x_4)\right). 
}$$
The group $G=\spn{r,t}\simeq \Z_4\times \Z_2$ acts on the product $W=S\times S$. Moreover, the generic invariant section of the anticanonical bundle of $W$ is smooth and the fixed locus of $G$ is disjoint from it. Hence, the zero locus $T$ of such a section is a smooth Calabi-Yau threefold and admits a free action of $G$. Call $H$ one of the cyclic subgroups of $G$ of order $4$. $H$ acts freely on $T$ as well.
\vspace{2mm}

\noindent Our interest in such variety is that the quotient $T/H$ is a smooth Calabi-Yau with Hodge numbers $(4,8)$. This implies that the Calabi-Yau $Z$ of Theorem (\ref{THM:RESOLUTION}) yields an explicit example of a Hodge-theoretic mirror for $T/G$. Both these varieties have the same fundamental group.
\end{remark}

\section{Hodge Theoretic mirrors for some non-maximal admissible pairs in $(\P^1)^4$}
\label{SEC:MIR}

In the previous section we have seen that if we start from a maximal admissible pair $(Y,G)$ in $X$, we are not able to find an admissible pair $(\ti{Y},\ti{G})$ in any $\ti{X}$ obtained from a maximal projective subdivision of $\du{\Sigma}$ such that $Z=Y/G$ and $\ti{Z}=\ti{Y}/\ti{G}$ are Hodge-Theoretic mirror Calabi-Yau. If we, instead, start from a non-maximal admissible pair this is possible in some cases.
\vspace{2mm}

\noindent We recall that in (\cite{BFNP13}) it is shown that two groups that are part of two admissible pairs in $X$ are isomorphic if and only if they are conjugated in $\Aut(X)$. So there is no ambiguity by writing $(Y,\Z_8)$. It is also shown that the quotients associated to the admissible pairs $(Y,\Z_8),(Y,\Z_4)$ and $(Y,\Z_2)$ in $X$ have respectively Hodge numbers $(1,9)$, $(2,18)$ and $(4,36)$. For these quotients, each fundamental group is isomorphic to the group of the admissible pair. We have all the information to prove the following theorem.

\begin{thm}
\label{THM:MIRROR2}
Let $\ti{X}$ denote the flag-resolution of $\du{X}$ and let $\ti{Y}$ be a smooth Calabi-Yau that is $\spn{g}-$invariant (such is the one used in Theorem (\ref{THM:MIRROR})). Then $(\ti{Y},\spn{g})$, $(\ti{Y},\spn{g^2})$ and $(\ti{Y},\spn{g}^4)$ are admissible pairs in $\ti{X}$ whose associated quotients satisfy:
\begin{itemize}
\item $\ti{Z}_0:=\ti{Y}/\spn{g}=\ti{Y}/\Z_8$ has Hodge-pair $(9,1)$;
\item $\ti{Z}_1:=\ti{Y}/\spn{g^2}=\ti{Y}/\Z_4$ has Hodge-pair $(18,2)$;
\item $\ti{Z}_2:=\ti{Y}/\spn{g^4}=\ti{Y}/\Z_2$ has Hodge-pair $(36,4)$.
\end{itemize}
In particular, $\ti{Z_0},\ti{Z_1}$ and $\ti{Z_2}$ are respectively Hodge-Theoretic mirrors for the quotients associated to the admissible pairs $(Y,\Z_8),(Y,\Z_4)$ and $(Y,\Z_2)$ in $X$. Each pairs of mirrors has the same fundamental group. 
\end{thm}

\begin{proof}
We have already shown that $\ti{Y}$ is a smooth Calabi-Yau threefold which is invariant with respect to $L_1=\spn{g,h_1}$ and that $g$, an automorphism of order $8$, acts freely on $\ti{Y}$ so that $(\ti{Y},\spn{g^{2^i}})$ is an admissible pair in $\ti{X}$ for $i\in\{0,1,2\}$. To obtain 
$$\h^{1,1}(\Z_i)=\h^{1,1}(\ti{Y})^{\spn{g^{2^i}}}$$ we adapt the same strategy as that used in Section (\ref{SEC:CHOSING}). Since the action of the group is free, we get
$$\chi(\ti{Z}_i)=128/\left|\spn{g^{2^i}}\right|=16\cdot 2^{i}$$
and $\h^{1,2}(\ti{Z_i})=\h^{1,1}(\Z_i)-8\cdot 2^{i}$. As noted in the proof of Theorem (\ref{THM:MIRROR}), $\ti{Y}$ is simply connected. The claim about the fundalmental group follows from the fact that the quotients are free.
\end{proof}




\end{document}